\documentclass[11pt]{article}
\usepackage[utf8]{inputenc}
\usepackage[T1]{fontenc}
\usepackage{lmodern}
\usepackage{amsmath,amsfonts,amssymb,mathtools,amsthm}
\usepackage[a4paper,margin=1in]{geometry}
\usepackage{microtype}
\usepackage{setspace}
\usepackage{needspace}
\newtheorem{theorem}{Theorem}[section]
\newtheorem{lemma}[theorem]{Lemma}
\newtheorem{proposition}[theorem]{Proposition}

\theoremstyle{definition}
\newtheorem{definition}[theorem]{Definition}
\theoremstyle{remark}

\newcommand{\R}{\mathbb{R}}
\newcommand{\C}{\mathbb{C}}
\newcommand{\e}{\mathrm{e}}
\newcommand{\supp}{\operatorname{supp}}
\newcommand{\Nmap}{\mathcal N}
\newcommand{\wt}{\widetilde}
\title{Logarithmic Stability for an Inverse Source Problem in a Coupled Nonlinear Helmholtz System}
\author{Hongcan Ye \and Lei Zhang \and Ting Zhou}
\date{}

\begin{document}
\maketitle

\begin{abstract}
We consider an inverse source problem for a coupled nonlinear
Helmholtz system motivated by second-harmonic generation. For a known
susceptibility bounded away from zero, we prove that small real-valued
sources are uniquely determined by the full first Fr\'{e}chet derivative
of the nonlinear Dirichlet-to-Neumann map at a fixed small real boundary
value. Under an a priori Sobolev bound, we derive a conditional
logarithmic stability estimate. The proof uses linearization, complex
geometric optics solutions for a matrix Schr\"odinger equation, and
interpolation.
\end{abstract}



\section{Introduction}

Let $\Omega\subset\R^3$ be a bounded domain with smooth boundary,
and let $\omega>0$. We consider the coupled nonlinear Helmholtz
system
\begin{equation}\label{eq:coupled_system}
\begin{cases}
(-\Delta-\omega^2)u_\omega+2a(x)u_{2\omega}u_\omega=F
&\text{in }\Omega,\\
(-\Delta-4\omega^2)u_{2\omega}+a(x)u_\omega^2=0
&\text{in }\Omega,\\
u_\omega=f_\omega,\quad u_{2\omega}=f_{2\omega}
&\text{on }\partial\Omega.
\end{cases}
\end{equation}
Here $a$ is a known susceptibility coefficient, $F$ is the unknown
source, and $f=(f_\omega,f_{2\omega})$ is the Dirichlet boundary value.
The coefficient and data are real-valued. We consider the small
solution of \eqref{eq:coupled_system} and study the stable determination
of $F$ from linearized boundary measurements at a fixed frequency.

The coupling in \eqref{eq:coupled_system} is motivated by
second-harmonic generation; see
\cite{Boyd2020,BaoDobson1994}. Consider the formal scalar model
\begin{equation}\label{eq:wave_eq}
(\partial_t^2-\Delta)u+a(x)u^2
=F(x)\e^{-i\omega t}+\overline{F(x)}\e^{i\omega t}.
\end{equation}
For the two-mode ansatz
\begin{equation}\label{eq:ansatz}
u=u_\omega\e^{-i\omega t}+\overline{u_\omega}\e^{i\omega t}
+u_{2\omega}\e^{-2i\omega t}
+\overline{u_{2\omega}}\e^{2i\omega t},
\end{equation}
the coefficients of $\e^{-i\omega t}$ and $\e^{-2i\omega t}$
in $u^2$ are $2u_{2\omega}\overline{u_\omega}$ and $u_\omega^2$,
respectively. Keeping only the terms at $\omega$ and $2\omega$ in
\eqref{eq:wave_eq} and taking the amplitudes to be real, we obtain
\eqref{eq:coupled_system}.

Under the nonresonance assumptions below, the direct problem is
locally well-posed for small source and boundary data. For a fixed
$F$, define the nonlinear Dirichlet-to-Neumann (DtN) map by
\begin{equation}\label{eq:nonlinear_dtn}
\Lambda_{a,F}(f)
=\bigl(\partial_\nu u_\omega,\partial_\nu u_{2\omega}\bigr)
\big|_{\partial\Omega},
\end{equation}
where $\nu$ is the outer unit normal to $\partial\Omega$.
We determine $F$ from the first Fr\'{e}chet derivative of this map
at a fixed boundary value $f^0$.

Fix $0<\gamma<1$. Assume that
\begin{equation}\label{eq:a_assumptions}
a\in C^{2,\gamma}(\overline\Omega;\R),
\qquad 0<a_0\le a(x)\le a_1
\quad\text{in }\overline\Omega.
\end{equation}
We also assume that $\omega^2$ and $4\omega^2$ are not Dirichlet
eigenvalues of $-\Delta$ in $\Omega$, and write
\begin{equation}\label{eq:nonresonance_gaps}
d_1=\operatorname{dist}(\omega^2,\sigma_D(-\Delta))>0,
\qquad
d_2=\operatorname{dist}(4\omega^2,\sigma_D(-\Delta))>0,
\end{equation}
where $\sigma_D(-\Delta)$ denotes the Dirichlet spectrum.

Let $\varepsilon_{\mathrm{wp}}$ be the constant in
Lemma~\ref{lem:well_posedness}, and choose
$0<\varepsilon_*\le\varepsilon_{\mathrm{wp}}/2$ as in
Proposition~\ref{prop:sobolev_linearized_map}. Fix
$f^0\in(C^{2,\gamma}(\partial\Omega;\R))^2$ such that
\begin{equation}\label{eq:base_boundary_smallness}
\|f^0_\omega\|_{C^{2,\gamma}(\partial\Omega)}
+\|f^0_{2\omega}\|_{C^{2,\gamma}(\partial\Omega)}
\le\frac{\varepsilon_*}{2}.
\end{equation}
For each $F\in C^\gamma(\overline\Omega;\R)$ with
$\|F\|_{C^\gamma}\le\varepsilon_*/2$, let
$U_F^0=(u^0_{\omega,F},u^0_{2\omega,F})$ be the corresponding
small solution. For $h\in(C^{2,\gamma}(\partial\Omega;\R))^2$, the
linearized solution $V_{F,h}$ in the boundary direction $h$ satisfies
\begin{equation}\label{eq:linearized_system}
(P+\mathbb Q_F)V_{F,h}=0\quad\text{in }\Omega,
\qquad V_{F,h}|_{\partial\Omega}=h,
\end{equation}
where
\begin{equation}\label{eq:linearized_potential}
P=\begin{pmatrix}
-\Delta-\omega^2&0\\
0&-\Delta-4\omega^2
\end{pmatrix},
\qquad
\mathbb Q_F=2a
\begin{pmatrix}
u^0_{2\omega,F}&u^0_{\omega,F}\\
u^0_{\omega,F}&0
\end{pmatrix}.
\end{equation}
The linearized DtN map is
\begin{equation}\label{eq:linearized_map_definition}
\Nmap_F=D_f\Lambda_{a,F}(f^0),
\qquad
\Nmap_Fh=\partial_\nu V_{F,h}|_{\partial\Omega}.
\end{equation}
Proposition~\ref{prop:sobolev_linearized_map} gives a bounded extension
to Sobolev boundary values. We use the same notation for its
complexification and set
\[
\|T\|_*=
\|T\|_{\mathcal L(
(H^{1/2}(\partial\Omega;\C))^2,
(H^{-1/2}(\partial\Omega;\C))^2)}.
\]

Our first result is the following uniqueness theorem.

\begin{theorem}\label{thm:source_uniqueness}
Assume \eqref{eq:a_assumptions} and \eqref{eq:nonresonance_gaps},
and let the common boundary value $f^0$ satisfy
\eqref{eq:base_boundary_smallness}. Let
$F_1,F_2\in C^\gamma(\overline\Omega;\R)$ satisfy
\[
\|F_j\|_{C^\gamma(\overline\Omega)}\le\varepsilon_*/2,
\qquad j=1,2.
\]
Let the linearized boundary maps correspond to the small real
solutions. If
\[
\Nmap_{F_1}=\Nmap_{F_2},
\]
then $F_1=F_2$ in $\Omega$.
\end{theorem}

The proof is given in Section~\ref{sec:uniqueness}. To study stability,
we introduce the following admissible class.

\begin{definition}\label{def:admissible_class}
For $M>0$ and $s>0$, let
\begin{equation}\label{eq:admissible_class}
\mathcal A_{M,s}=
\left\{
F\in C^\gamma(\overline\Omega;\R)\cap H^s(\Omega):
\begin{array}{l}
\|F\|_{C^\gamma(\overline\Omega)}\le\varepsilon_*/2,\\
\|F\|_{H^s(\Omega)}\le M
\end{array}
\right\}.
\end{equation}

\end{definition}

Our stability result is stated as follows.

\begin{theorem}\label{thm:main_theorem}
Assume \eqref{eq:a_assumptions} and \eqref{eq:nonresonance_gaps},
and let the common boundary value $f^0$ satisfy
\eqref{eq:base_boundary_smallness}. Fix $M,s>0$ and $0<r<3/2$.
There exist constants $C>0$ and $\delta_0\in(0,1)$ such that for all
$F_1,F_2\in\mathcal A_{M,s}$ with
\[
0<\delta:=\|\Nmap_{F_2}-\Nmap_{F_1}\|_*<\delta_0,
\]
we have
\begin{equation}\label{eq:main_stability}
\|F_2-F_1\|_{L^2(\Omega)}\le C|\log\delta|^{-\kappa},
\qquad
\kappa=\frac{s}{s+2}\frac{2r}{3+2r}.
\end{equation}
The constants $C$ and $\delta_0$ are uniform over all pairs
$F_1,F_2\in\mathcal A_{M,s}$.
\end{theorem}

Both theorems use the full boundary operator $\Nmap_F$ at a fixed
$\omega$ and a common $f^0$. The two Dirichlet components can be
perturbed independently, and both normal derivatives are measured.
Smallness of the data ensures uniform solvability of the direct and
linearized problems.

For linear equations, a source cannot in general be recovered from
boundary Cauchy data at a fixed frequency. Indeed, for
$(-\Delta-\omega^2)u=F$, replacing $u$ by $u+\psi$ and $F$ by
$F+(-\Delta-\omega^2)\psi$, where $\psi\in C_c^\infty(\Omega)$,
leaves the boundary Cauchy data unchanged.
Multi-frequency inverse source and scattering problems have been
studied in
\cite{BaoLinTriki2010,BaoLiLinTriki2015,ChengIsakovLu2016};
see also \cite{IsakovBook}.

A standard approach to inverse problems for nonlinear elliptic
equations is to linearize the boundary map and use results for
the resulting linear equations; see
\cite{Lassas2021,Feizmohammadi2020}.
Increasing stability for linearized inverse potential problems with
power type nonlinearities was studied in \cite{LuSaloXu2022}.
For scalar semilinear elliptic equations, Liimatainen and Lin
\cite{Liimatainen2024} proved uniqueness results and characterized
gauge invariances for the simultaneous recovery of sources and
nonlinearities by linearizing around possibly nonzero solutions.
In particular, their Remark~1.3
gives source uniqueness for quadratic nonlinearities when the
linear coefficient is known and the quadratic coefficient does
not vanish. When both coefficients are known, their argument in
Section~3.1 also implies that the first boundary derivative determines
the background solution through the linearized potential, and hence
the source.

The proof in the present paper uses the fact that the first boundary
derivative is the DtN map of a symmetric $2\times2$ matrix
Schr\"odinger operator. The entries
$(\mathbb Q_F)_{21}=2au^0_{\omega,F}$ and
$(\mathbb Q_F)_{11}=2au^0_{2\omega,F}$ determine both background
fields. We determine the matrix potential using an Alessandrini
identity and complex geometric optics (CGO) solutions, following
\cite{SylvesterUhlmann1987,Alessandrini1988}.
This gives uniqueness and a logarithmic stability estimate for the
background fields. For exponential instability in scalar inverse
potential problems, see \cite{Mandache2001}.

For the source estimate, we subtract the fundamental-frequency
equations and bound the source difference in $H^{-2}(\Omega)$.
Interpolation on $\Omega$ with the a priori $H^s(\Omega)$ bound
then gives the $L^2$ estimate. The common boundary value is used
in the background estimate: the differences of the background
solutions have zero Dirichlet trace, and their zero extensions
are uniformly bounded in $H^r(\R^3)$ for $0<r<3/2$.

The paper is organized as follows. In Section~2, we prove
well-posedness of the direct problem and study the linearized DtN
map. In Section~3, we recover the matrix potential and prove
Theorem~\ref{thm:source_uniqueness}. Section~4 contains the logarithmic
stability estimates for the background fields and the source.
The appendix gives the proof of the quantitative implicit function
theorem.

\paragraph{Acknowledgements.}
L.Z. is supported in part by the National Key R\&D Program of China
(No.~2024YFA1012300) and the National Natural Science Foundation of China
(No.~12271482). T.Z. is supported by the National Natural Science Foundation
of China (Grant No.~12371426) and the Basic Public Welfare Research Program
of Zhejiang Province (Grant No.~LDQ24A010001).

\section{Well-posedness and the Linearized Boundary Map}

We prove local well-posedness for small real data and derive the
linearized DtN map.

\subsection{A quantitative implicit function theorem}

We shall use the following quantitative implicit function theorem,
whose proof is given in Appendix~\ref{app:qift_proof}.

\begin{lemma}\label{lem:qift}
Let $X,Y,Z$ be Banach spaces, let $W\subset X\times Y$ be open,
and let $\mathcal F\in C^m(W;Z)$, $m\ge1$. Suppose that
$\mathcal F(x_0,y_0)=0$ and
$A=D_2\mathcal F(x_0,y_0):Y\to Z$ is boundedly invertible.
Let $r,\rho>0$ satisfy
$\overline B_r^X(x_0)\times\overline B_\rho^Y(y_0)\subset W$ and
\begin{align}
\sup_{x\in\overline B_r^X(x_0)}
\|A^{-1}\mathcal F(x,y_0)\|_Y&\le\rho/2,
\label{eq:qift_cond1}\\
\sup_{(x,y)\in\overline B_r^X(x_0)\times\overline B_\rho^Y(y_0)}
\|A^{-1}D_2\mathcal F(x,y)-I\|_{\mathcal L(Y)}&\le1/2.
\label{eq:qift_cond2}
\end{align}
Then for each $x\in\overline B_r^X(x_0)$ there is a unique
$u(x)\in\overline B_\rho^Y(y_0)$ such that
$\mathcal F(x,u(x))=0$. The map $u$ is continuous on the closed
parameter ball, satisfies $u(x_0)=y_0$, and is $C^m$ on its interior.
\end{lemma}

\subsection{The direct problem}

\begin{lemma}\label{lem:well_posedness}
Assume that $a\in C^{2,\gamma}(\overline\Omega;\R)$ and
that $\omega^2,4\omega^2$ are not Dirichlet eigenvalues of $-\Delta$.
There exist $\varepsilon_{\mathrm{wp}}>0$ and $\rho_*>0$ such that
for real data satisfying
\begin{equation}\label{eq:data_smallness}
\|F\|_{C^\gamma(\overline\Omega)}
+\|f_\omega\|_{C^{2,\gamma}(\partial\Omega)}
+\|f_{2\omega}\|_{C^{2,\gamma}(\partial\Omega)}
<\varepsilon_{\mathrm{wp}},
\end{equation}
the problem \eqref{eq:coupled_system} has a unique real solution
$U\in(C^{2,\gamma}(\overline\Omega))^2$ in the ball
$\|U\|_{(C^{2,\gamma})^2}\le\rho_*$. Moreover,
\begin{equation}\label{eq:apriori_est}
\|U\|_{(C^{2,\gamma})^2}
\le C\bigl(\|F\|_{C^\gamma}
+\|f_\omega\|_{C^{2,\gamma}}
+\|f_{2\omega}\|_{C^{2,\gamma}}\bigr).
\end{equation}
The solution depends smoothly on the data in this open set.
\end{lemma}

\begin{proof}
Consider the following real Banach spaces, equipped with the sum norms:
\begin{align*}
X&=C^\gamma(\overline\Omega;\R)
\times(C^{2,\gamma}(\partial\Omega;\R))^2,\\
Y&=(C^{2,\gamma}(\overline\Omega;\R))^2,\\
Z&=(C^\gamma(\overline\Omega;\R))^2
\times(C^{2,\gamma}(\partial\Omega;\R))^2.
\end{align*}
For $D=(F,f_\omega,f_{2\omega})$ and $U=(u_\omega,u_{2\omega})$,
define
\begin{equation}\label{eq:operator_M}
\mathcal M(D,U)=
\begin{pmatrix}
(-\Delta-\omega^2)u_\omega+2au_{2\omega}u_\omega-F\\
(-\Delta-4\omega^2)u_{2\omega}+au_\omega^2\\
u_\omega|_{\partial\Omega}-f_\omega\\
u_{2\omega}|_{\partial\Omega}-f_{2\omega}
\end{pmatrix}.
\end{equation}
This is a smooth polynomial map, and $\mathcal M(D,U)=0$ is
equivalent to \eqref{eq:coupled_system}.

Its derivative with respect to the solution at the origin is
\[
AV:=D_2\mathcal M(0,0)V
=\bigl((-\Delta-\omega^2)v_\omega,
(-\Delta-4\omega^2)v_{2\omega},
v_\omega|_{\partial\Omega},v_{2\omega}|_{\partial\Omega}\bigr)^T.
\]
The nonresonance assumption gives unique solvability of the two
scalar Dirichlet problems. By the Schauder estimates
\cite{GilbargTrudinger}, we have
\begin{equation}\label{eq:inverse_bound_final}
\|A^{-1}\|_{\mathcal L(Z,Y)}\le C_0.
\end{equation}
The constant depends on the domain, $\gamma$, $\omega$ and the
spectral gaps. Setting the solution variable equal to zero gives
\[
\mathcal M(D,0)=(-F,0,-f_\omega,-f_{2\omega})^T.
\]
By \eqref{eq:inverse_bound_final},
\[
\|A^{-1}\mathcal M(D,0)\|_Y\le C_0\|D\|_X.
\]
The H\"older product estimate yields
\begin{equation}\label{eq:contraction_cond}
\|A^{-1}D_2\mathcal M(D,U)-I\|_{\mathcal L(Y)}
\le C_0C_A\|a\|_{C^\gamma}\|U\|_Y.
\end{equation}
Choose $\rho_*>0$ so that
$C_0C_A\|a\|_{C^\gamma}\rho_*\le1/2$, and then choose
$\varepsilon_{\mathrm{wp}}>0$ so that
$C_0\varepsilon_{\mathrm{wp}}\le\rho_*/2$.
Lemma~\ref{lem:qift} now gives existence and uniqueness in the stated ball.

To prove \eqref{eq:apriori_est}, write
\[
\mathcal M(D,U)=AU+\mathcal R(U)
-(F,0,f_\omega,f_{2\omega})^T,
\qquad
\mathcal R(U)=(2au_{2\omega}u_\omega,au_\omega^2,0,0)^T.
\]
Taking $C_A$ larger if necessary before choosing the radii, we obtain
\[
\|U\|_Y\le C_0\|D\|_X
+C_0C_A\|a\|_{C^\gamma}\|U\|_Y^2
\le C_0\|D\|_X+\tfrac12\|U\|_Y.
\]
This proves the estimate. By \eqref{eq:contraction_cond}, the derivative
with respect to the solution is invertible throughout the solution ball.
The implicit function theorem then gives a $C^\infty$ solution map
on the open data ball.
\end{proof}

Fix real data $(F,f^0)$ satisfying \eqref{eq:data_smallness}.
For $f$ in a neighborhood of $f^0$ in
$(C^{2,\gamma}(\partial\Omega;\R))^2$, denote the small solution of
\eqref{eq:coupled_system} by $\mathcal S_F(f)$. By
Lemma~\ref{lem:well_posedness}, the solution map $\mathcal S_F$ is
$C^\infty$ with values in $(C^{2,\gamma}(\overline\Omega;\R))^2$.
Write $U_F^0=\mathcal S_F(f^0)=(u^0_{\omega,F},u^0_{2\omega,F})$.

Let $h=(h_\omega,h_{2\omega})\in(C^{2,\gamma}(\partial\Omega;\R))^2$.
For sufficiently small real $t$, the perturbed data $(F,f^0+th)$
remain in the open set given by \eqref{eq:data_smallness}. Write
$U(t)=(u_\omega(t),u_{2\omega}(t))=\mathcal S_F(f^0+th)$.
The perturbed boundary value problem is
\begin{equation}\label{eq:boundary_perturbed_system}
\begin{cases}
(-\Delta-\omega^2)u_\omega(t)+2a u_{2\omega}(t)u_\omega(t)=F
&\text{in }\Omega,\\
(-\Delta-4\omega^2)u_{2\omega}(t)+a u_\omega(t)^2=0
&\text{in }\Omega,\\
u_\omega(t)=f_\omega^0+th_\omega,\quad
u_{2\omega}(t)=f_{2\omega}^0+th_{2\omega}
&\text{on }\partial\Omega.
\end{cases}
\end{equation}
For $t\ne0$, set
\[
W_t=(w_{\omega,t},w_{2\omega,t})=\frac{U(t)-U_F^0}{t}.
\]
The boundary value of $W_t$ is $h$. Substituting
$u_\omega(t)=u^0_{\omega,F}+tw_{\omega,t}$ and
$u_{2\omega}(t)=u^0_{2\omega,F}+tw_{2\omega,t}$ gives
\begin{align*}
\frac{u_{2\omega}(t)u_\omega(t)-u^0_{2\omega,F}u^0_{\omega,F}}{t}
&=u^0_{2\omega,F}w_{\omega,t}+u^0_{\omega,F}w_{2\omega,t}
  +t w_{2\omega,t}w_{\omega,t},\\
\frac{u_\omega(t)^2-(u^0_{\omega,F})^2}{t}
&=2u^0_{\omega,F}w_{\omega,t}+t w_{\omega,t}^2.
\end{align*}
Subtracting the equations at $t=0$ from \eqref{eq:boundary_perturbed_system}
and dividing by $t$ yields
\begin{equation}\label{eq:boundary_difference_quotient}
(P+\mathbb Q_F)W_t=-t
\begin{pmatrix}
2a w_{2\omega,t}w_{\omega,t}\\
a w_{\omega,t}^2
\end{pmatrix}
\quad\text{in }\Omega,
\qquad W_t|_{\partial\Omega}=h.
\end{equation}
The source terms cancel since $F$ is fixed. The dependence on $F$
in the linear terms is through $U_F^0$ and hence through $\mathbb Q_F$.

Since $\mathcal S_F$ is Fr\'{e}chet differentiable,
\[
W_t\longrightarrow V_{F,h}:=D_f\mathcal S_F(f^0)h
=\left.\frac{d}{dt}U(t)\right|_{t=0}
\quad\text{in }(C^{2,\gamma}(\overline\Omega;\R))^2.
\]
In particular, $W_t$ is bounded in this space, so the right-hand side
of \eqref{eq:boundary_difference_quotient} converges to zero in
$(C^\gamma(\overline\Omega;\R))^2$. Taking the limit in the equations
and the boundary condition, and writing
$V_h=V_{F,h}=(v_{\omega,h},v_{2\omega,h})$, we obtain
\[
\begin{cases}
(-\Delta-\omega^2)v_{\omega,h}
+2a\bigl(u^0_{2\omega,F}v_{\omega,h}
+u^0_{\omega,F}v_{2\omega,h}\bigr)=0
&\text{in }\Omega,\\
(-\Delta-4\omega^2)v_{2\omega,h}
+2au^0_{\omega,F}v_{\omega,h}=0
&\text{in }\Omega,\\
V_h=h&\text{on }\partial\Omega.
\end{cases}
\]
The coefficients of $v_{\omega,h}$ and $v_{2\omega,h}$ in the coupling
terms give the two rows of $\mathbb Q_F$ in
\eqref{eq:linearized_potential}. Thus the component equations give
\eqref{eq:linearized_system}.

To identify the derivative of the DtN map, let $k$ be an arbitrary
small real increment in $(C^{2,\gamma}(\partial\Omega;\R))^2$.
Taylor's formula for the $C^2$ map $\mathcal S_F$ gives
\[
\mathcal S_F(f^0+k)=U_F^0+V_{F,k}+R_F(k),
\qquad
\|R_F(k)\|_{(C^{2,\gamma}(\overline\Omega))^2}
\le C\|k\|_{(C^{2,\gamma}(\partial\Omega))^2}^2,
\]
where $C$ can be chosen for a neighborhood of the fixed $f^0$.
Applying the bounded linear normal trace from
$(C^{2,\gamma}(\overline\Omega))^2$ to
$(C^{1,\gamma}(\partial\Omega))^2$ gives
\[
\left\|\Lambda_{a,F}(f^0+k)-\Lambda_{a,F}(f^0)
-\left.\partial_\nu V_{F,k}\right|_{\partial\Omega}
\right\|_{(C^{1,\gamma}(\partial\Omega))^2}
\le C\|k\|_{(C^{2,\gamma}(\partial\Omega))^2}^2.
\]
Consequently, in $(C^{1,\gamma}(\partial\Omega;\R))^2$,
\begin{align*}
D_f\Lambda_{a,F}(f^0)h
&=\lim_{t\to0}\frac{\Lambda_{a,F}(f^0+th)-\Lambda_{a,F}(f^0)}{t}\\
&=\left.(\partial_\nu v_{\omega,h},\partial_\nu v_{2\omega,h})
\right|_{\partial\Omega}.
\end{align*}
This proves \eqref{eq:linearized_map_definition} on real H\"older
boundary data.

\subsection{The linearized DtN map}

\begin{proposition}\label{prop:sobolev_linearized_map}
There exists $0<\varepsilon_*\le\varepsilon_{\mathrm{wp}}/2$ such that,
for real data satisfying
$\|F\|_{C^\gamma}+\|f^0\|_{(C^{2,\gamma})^2}\le\varepsilon_*$,
problem \eqref{eq:linearized_system} has a unique weak solution
$V_h\in H^1(\Omega;\C^2)$ for every
$h\in(H^{1/2}(\partial\Omega;\C))^2$, with
\begin{equation}\label{eq:weak_linearized_solution_bound}
\|V_h\|_{H^1(\Omega;\C^2)}
\le C\|h\|_{(H^{1/2}(\partial\Omega;\C))^2}.
\end{equation}
Its weak normal trace defines a bounded operator
\[
\Nmap_F:(H^{1/2}(\partial\Omega;\C))^2
\longrightarrow (H^{-1/2}(\partial\Omega;\C))^2
\]
that agrees with $D_f\Lambda_{a,F}(f^0)$ on
$(C^{2,\gamma}(\partial\Omega;\R))^2$.
Both bounds are uniform in the data.
\end{proposition}

\begin{proof}
Let $P_D:H_0^1(\Omega;\C^2)\to H^{-1}(\Omega;\C^2)$ be the
homogeneous Dirichlet realization of $P$. By nonresonance and the
Fredholm alternative, $P_D$ is boundedly invertible. Multiplication
by $\mathbb Q_F$ is bounded from $H_0^1(\Omega;\C^2)$ to
$H^{-1}(\Omega;\C^2)$ with norm at most
$C\|\mathbb Q_F\|_{L^\infty}$. By \eqref{eq:apriori_est},
$\|\mathbb Q_F\|_{L^\infty}\le C\varepsilon_*$, so we can choose
$\varepsilon_*$ such that
$\|P_D^{-1}\mathbb Q_F\|_{\mathcal L(H_0^1(\Omega;\C^2))}\le 1/2$.
A Neumann series gives a uniform bound for
$(P_D+\mathbb Q_F)^{-1}:H^{-1}(\Omega;\C^2)\to H_0^1(\Omega;\C^2)$.

Let $E:(H^{1/2}(\partial\Omega;\C))^2\to H^1(\Omega;\C^2)$ be
a bounded right inverse of the trace. Then $V_h=Eh+w$, where
\[
w=-(P_D+\mathbb Q_F)^{-1}(P+\mathbb Q_F)Eh,
\]
is the unique weak solution and satisfies
\eqref{eq:weak_linearized_solution_bound}. With
$\mathsf K=\operatorname{diag}(\omega^2,4\omega^2)$, define
\begin{equation}\label{eq:weak_linearized_dtn}
\langle\Nmap_Fh,g\rangle_{\partial\Omega}
=\int_\Omega\left[
\sum_{p=1}^2\nabla(V_h)_p\cdot\nabla(Eg)_p
+((\mathbb Q_F-\mathsf K)V_h)\cdot Eg
\right]\,dx.
\end{equation}
Here $g\in(H^{1/2}(\partial\Omega;\C))^2$ and the pairings are
bilinear. The weak equation gives independence of the lifting of
$g$, and \eqref{eq:weak_linearized_solution_bound} gives
$|\langle\Nmap_Fh,g\rangle_{\partial\Omega}|
\le C\|h\|_{H^{1/2}}\|g\|_{H^{1/2}}$.

For real $h\in(C^{2,\gamma}(\partial\Omega;\R))^2$, weak
uniqueness identifies $V_h$ with the derivative of the classical
solution. Green's formula identifies \eqref{eq:weak_linearized_dtn}
with its classical normal trace. Smooth real functions are dense in
$H^{1/2}(\partial\Omega;\R)$, so the real restriction of $\Nmap_F$ is
the unique bounded Sobolev extension of the measured derivative.
Since the coefficients are real, the complex operator is the
complexification of this real Sobolev operator. With the standard
Hilbert complexification norms, their operator norms are equal, as
are the norms of differences of such operators.
\end{proof}

\section{Uniqueness for the Source}\label{sec:uniqueness}

Let $F_1,F_2$ be real sources in
$C^\gamma(\overline\Omega)$ satisfying
$\|F_j\|_{C^\gamma}\le\varepsilon_*/2$, and let the common
boundary value $f^0$ satisfy
\eqref{eq:base_boundary_smallness}.
For $j=1,2$, denote by $U_j^0=(u^0_{\omega,j},u^0_{2\omega,j})$ the
solution corresponding to $F_j$ and
$f^0$. Set
\begin{equation}\label{eq:Qj_definition}
\mathcal L_j=P+\mathbb Q_j,
\qquad
\mathbb Q_j=2a
\begin{pmatrix}u^0_{2\omega,j}&u^0_{\omega,j}\\
u^0_{\omega,j}&0\end{pmatrix}.
\end{equation}
The matrices $\mathbb Q_j$ are real symmetric and uniformly bounded
in $C^{2,\gamma}(\overline\Omega)$. Set
\begin{equation}\label{eq:total_matrix_potential}
\mathsf K=\operatorname{diag}(\omega^2,4\omega^2),
\qquad \mathbb V_j=\mathbb Q_j-\mathsf K.
\end{equation}
Then $\mathcal L_j=-\Delta I_2+\mathbb V_j$ and
$\mathbb V_2-\mathbb V_1=\mathbb Q_2-\mathbb Q_1$.
We shall construct CGO solutions for this operator with null phase
vectors and coordinate vectors as leading amplitudes.

We use the Fourier transform and Sobolev norms
\begin{equation}\label{eq:sobolev_fourier_norm}
\widehat G(\xi)=\int_{\R^3}\e^{-ix\cdot\xi}G(x)\,dx,
\qquad
\|G\|_{H^t(\R^3)}^2=(2\pi)^{-3}
\int_{\R^3}\langle\xi\rangle^{2t}|\widehat G(\xi)|^2\,d\xi,
\end{equation}
where $\langle\xi\rangle=(1+|\xi|^2)^{1/2}$.

\subsection{Alessandrini identity}

\begin{lemma}\label{lem:alessandrini_identity}
Let $V_j\in H^1(\Omega;\C^2)$ satisfy $\mathcal L_jV_j=0$, and
let $g_j=V_j|_{\partial\Omega}$. Then
\begin{equation}\label{eq:Alessandrini_identity_final}
\int_\Omega(\mathbb Q_2-\mathbb Q_1)V_1\cdot V_2\,dx
=\langle(\Nmap_{F_2}-\Nmap_{F_1})g_1,g_2\rangle_{\partial\Omega}.
\end{equation}
\end{lemma}

\begin{proof}
Taking $V_1$ as a lifting of $g_1$ and $V_2$ as a lifting of $g_2$
in \eqref{eq:weak_linearized_dtn} gives
\[
\int_\Omega(\mathbb Q_2-\mathbb Q_1)V_1\cdot V_2\,dx
=\langle\Nmap_{F_2}g_2,g_1\rangle_{\partial\Omega}
-\langle\Nmap_{F_1}g_1,g_2\rangle_{\partial\Omega}.
\]
Since $\mathbb Q_2^T=\mathbb Q_2$, the DtN map $\Nmap_{F_2}$ is
symmetric with respect to this pairing. The first term on the
right-hand side is therefore
$\langle\Nmap_{F_2}g_1,g_2\rangle_{\partial\Omega}$.
This proves the identity; see \cite{Alessandrini1988}.
\end{proof}

\subsection{CGO solutions}

Fix $-1<\delta_w<0$. We use the weighted spaces on $\R^3$ with norms
\[
\|v\|_{L^2_{\delta_w}}=\|\langle x\rangle^{\delta_w}v\|_{L^2},
\qquad
\|v\|_{H^1_{\delta_w}}=\|\langle x\rangle^{\delta_w}v\|_{H^1}.
\]
Choose a smooth bounded domain $\Omega_1$ with
$\overline\Omega\subset\Omega_1$. A fixed extension operator and a
cutoff function give real symmetric matrix fields $\wt{\mathbb V}_j$
satisfying
\begin{equation}\label{eq:V_extension_bound}
\wt{\mathbb V}_j|_\Omega=\mathbb V_j,
\qquad \supp\wt{\mathbb V}_j\subset\overline\Omega_1,
\qquad \|\wt{\mathbb V}_j\|_{L^\infty(\R^3)}\le C.
\end{equation}
Here the full potential $\mathbb Q_j-\mathsf K$ is extended with
compact support.

We recall the weighted estimates of Sylvester and Uhlmann
\cite{SylvesterUhlmann1987}; see also
\cite[Section~2]{CaroOlaSalo2009}. If
$\zeta\cdot\zeta=0$ and $|\operatorname{Re}\zeta|=\tau$ is
sufficiently large, the weighted inverse $G_\zeta$ of
$-\Delta-2\zeta\cdot\nabla$ satisfies
\begin{equation}\label{eq:null_phase_faddeev_estimates}
\|G_\zeta f\|_{L^2_{\delta_w}}
\le C\tau^{-1}\|f\|_{L^2_{\delta_w+1}},
\qquad
\|G_\zeta f\|_{H^1_{\delta_w}}
\le C\|f\|_{L^2_{\delta_w+1}}.
\end{equation}
The same notation denotes its componentwise action on
$\C^2$-valued functions. Let $e_1,e_2$ denote the standard basis of
$\C^2$.

\Needspace{10\baselineskip}
\begin{lemma}\label{lem:matrix_cgo}
There exist uniform constants $\tau_0,C>0$ such that, for
$m\in\{1,2\}$ and $\zeta\in\C^3$ satisfying
\[
\zeta\cdot\zeta=0,
\qquad |\operatorname{Re}\zeta|=\tau\ge\tau_0,
\]
there is a solution of $\mathcal L_jV=0$ in $\Omega$ of the form
\begin{equation}\label{eq:matrix_CGO}
V=\e^{\zeta\cdot x}(e_m+\Phi),
\end{equation}
where
\begin{equation}\label{eq:Phi_decay_bounds}
\|\Phi\|_{L^2(\Omega)}\le C\tau^{-1},
\qquad \|\Phi\|_{H^1(\Omega)}\le C.
\end{equation}
\end{lemma}

\begin{proof}
Since $\zeta\cdot\zeta=0$, conjugating the extended equation by
$\e^{\zeta\cdot x}$ gives
\[
(-\Delta-2\zeta\cdot\nabla)\Phi
=-\wt{\mathbb V}_j(e_m+\Phi).
\]
We seek a solution of
\begin{equation}\label{eq:cgo_correction}
\Phi=-G_\zeta\wt{\mathbb V}_j(e_m+\Phi).
\end{equation}
By \eqref{eq:V_extension_bound},
\[
\|\wt{\mathbb V}_j Z\|_{L^2_{\delta_w+1}}
\le C\|Z\|_{L^2_{\delta_w}},
\qquad
\|\wt{\mathbb V}_je_m\|_{L^2_{\delta_w+1}}\le C.
\]
Using \eqref{eq:null_phase_faddeev_estimates}, we obtain
$\|G_\zeta\wt{\mathbb V}_j\|_{\mathcal L(L^2_{\delta_w})}
\le C/\tau$. Choose $\tau_0\ge1$ so that this norm is at most
$1/2$. The Neumann series then gives
$\|\Phi\|_{L^2_{\delta_w}}\le C/\tau$.
Applying the $H^1$ estimate to \eqref{eq:cgo_correction}, we obtain
$\|\Phi\|_{H^1_{\delta_w}}\le C$.
Restriction to $\Omega$ gives \eqref{eq:Phi_decay_bounds}.
The correction equation also gives $\mathcal L_jV=0$ in the
distributional sense in $\Omega$.
\end{proof}

\subsection{Recovery of the matrix potential}

Set $\Delta\mathbb Q=\mathbb Q_2-\mathbb Q_1$. In what follows,
the Fourier transforms of its entries are taken after restriction to
$\Omega$ and extension by zero. These zero extensions are distinct
from the extensions used to construct the CGO solutions.
Fix $R>0$ such that $\Omega\subset B_R(0)$.

\begin{lemma}\label{lem:fourier_estimate}
For every $\ell,m\in\{1,2\}$, every $\tau\ge\tau_0$, and every
$\xi\in\R^3$ with $|\xi|\le\tau$,
\begin{equation}\label{eq:fourier_bound_Q}
|\widehat{(\Delta\mathbb Q)_{m\ell}}(\xi)|
\le C\bigl(\tau^{-1}+\tau^2\e^{2R\tau}\delta\bigr)
=:E(\tau,\delta),
\end{equation}
where $\delta=\|\Nmap_{F_2}-\Nmap_{F_1}\|_*\ge0$.
The estimate includes the case $\delta=0$.
\end{lemma}

\begin{proof}
For $\xi\ne0$, choose orthonormal vectors $\eta,\theta$ in
$\xi^\perp$. If $\xi=0$, take any two orthonormal vectors.
Set $\alpha=(\tau^2-|\xi|^2/4)^{1/2}$ and define
\begin{align}\label{eq:null_phases}
\zeta_1&=\tau\eta+i(\alpha\theta-\xi/2),\nonumber\\
\zeta_2&=-\tau\eta+i(-\alpha\theta-\xi/2).
\end{align}
These phase vectors satisfy, for $|\xi|\le\tau$,
\begin{equation}\label{eq:phase_relations}
\zeta_1+\zeta_2=-i\xi,
\qquad \zeta_1\cdot\zeta_1=\zeta_2\cdot\zeta_2=0,
\qquad |\zeta_1|=|\zeta_2|=\sqrt2\,\tau.
\end{equation}
By Lemma~\ref{lem:matrix_cgo}, there exist
$V_1=\e^{\zeta_1\cdot x}(e_\ell+\Phi_1)$ and
$V_2=\e^{\zeta_2\cdot x}(e_m+\Phi_2)$ solving
$\mathcal L_1V_1=0$ and $\mathcal L_2V_2=0$, respectively.
Substituting these solutions into
\eqref{eq:Alessandrini_identity_final}, we obtain
\begin{equation}\label{eq:fourier_identity_revised}
\widehat{(\Delta\mathbb Q)_{m\ell}}(\xi)
=\langle(\Nmap_{F_2}-\Nmap_{F_1})g_1,g_2\rangle_{\partial\Omega}
+\mathcal R,
\end{equation}
where the uniform $L^\infty$ bound for $\Delta\mathbb Q$ gives
\[
|\mathcal R|\le C\bigl(\|\Phi_1\|_{L^2}+\|\Phi_2\|_{L^2}
+\|\Phi_1\|_{L^2}\|\Phi_2\|_{L^2}\bigr)\le C/\tau.
\]
By \eqref{eq:Phi_decay_bounds} and the bounds for the phase vectors,
\[
\|V_j\|_{H^1(\Omega)}
\le C\e^{R\tau}\bigl[(1+|\zeta_j|)(1+\|\Phi_j\|_{L^2})
+\|\nabla\Phi_j\|_{L^2}\bigr]
\le C\tau\e^{R\tau}.
\]
The trace theorem gives the same bound for
$\|g_j\|_{H^{1/2}(\partial\Omega)}$. Thus, the boundary term in
\eqref{eq:fourier_identity_revised} is bounded by
$C\tau^2\e^{2R\tau}\delta$, and the lemma follows.
\end{proof}

\begin{proof}[Proof of Theorem~\ref{thm:source_uniqueness}]
Suppose that $\Nmap_{F_1}=\Nmap_{F_2}$, so that $\delta=0$.
Fix $\xi\in\R^3$. Taking $\tau$ sufficiently large that
$|\xi|\le\tau$, we obtain from Lemma~\ref{lem:fourier_estimate} that
\[
\bigl|\widehat{(\mathbb Q_2-\mathbb Q_1)_{m\ell}}(\xi)\bigr|
\le C/\tau,
\qquad m,\ell\in\{1,2\}.
\]
Letting $\tau\to\infty$, we find that the Fourier transform of
each zero-extended entry of $\mathbb Q_2-\mathbb Q_1$ vanishes.
These entries belong to $L^1(\R^3)$ since $\Omega$ is bounded.
The injectivity of the Fourier transform gives $\mathbb Q_1=\mathbb Q_2$
in $\Omega$. By \eqref{eq:Qj_definition},
\[
2a\bigl(u^0_{\omega,2}-u^0_{\omega,1}\bigr)=0,
\qquad
2a\bigl(u^0_{2\omega,2}-u^0_{2\omega,1}\bigr)=0.
\]
Since $a\ge a_0>0$, the two background solutions coincide.
The first equation of \eqref{eq:coupled_system} now gives
$F_1=F_2$ in $\Omega$.
\end{proof}

\section{Logarithmic Stability}\label{sec:stability}

\subsection{Stability of the background fields}

We first estimate the background fields under the assumptions of
Section~\ref{sec:uniqueness}.

\begin{lemma}\label{lem:linearized_stability}
Let $0<r<3/2$ and $\sigma_r=2r/(3+2r)$. There exist uniform
constants $C>0$ and $0<\delta_0<1$ such that, for
$0<\delta=\|\Nmap_{F_2}-\Nmap_{F_1}\|_*<\delta_0$,
\begin{equation}\label{eq:background_stability}
\|u^0_{\omega,2}-u^0_{\omega,1}\|_{L^2(\Omega)}
+\|u^0_{2\omega,2}-u^0_{2\omega,1}\|_{L^2(\Omega)}
\le C|\log\delta|^{-\sigma_r}.
\end{equation}
The constants  are uniform over all source pairs
$F_1,F_2$ satisfying the assumptions of
Theorem~\ref{thm:source_uniqueness}.
\end{lemma}

For a function $v$ on $\Omega$, let $E_0v$ denote its zero
extension to $\R^3$.

\begin{lemma}\label{lem:zero_extension}
For every $0<r<3/2$,
\[
\|E_0v\|_{H^r(\R^3)}\le C_r\|v\|_{H^2(\Omega)},
\qquad v\in H^2(\Omega)\cap H_0^1(\Omega).
\]
\end{lemma}

\begin{proof}
Zero extension is bounded from $H_0^1(\Omega)$ into $H^1(\R^3)$,
which proves the estimate for $r\le1$. For $1<r<3/2$, the zero
Dirichlet trace gives $\partial_iE_0v=E_0(\partial_i v)$ in
distributions. The zero-extension estimate below order $1/2$
\cite{AdamsFournier}, applied to $\partial_i v\in H^1(\Omega)$,
gives $E_0(\partial_i v)\in H^{r-1}(\R^3)$ with the stated bound.
\end{proof}

\begin{proof}[Proof of Lemma~\ref{lem:linearized_stability}]
Put
\[
\wt u_\omega=u^0_{\omega,2}-u^0_{\omega,1},
\qquad \wt u_{2\omega}=u^0_{2\omega,2}-u^0_{2\omega,1}.
\]
The entries of the potential difference satisfy
\begin{equation}\label{eq:q_entries}
q_1=a\wt u_\omega=\tfrac12(\Delta\mathbb Q)_{21},
\qquad q_2=a\wt u_{2\omega}=\tfrac12(\Delta\mathbb Q)_{11}.
\end{equation}
The background solutions have the same boundary value, so their
differences have zero Dirichlet trace. They also have a uniform
$H^2(\Omega)$ bound. These properties also hold for $q_1,q_2$, since
$a\in C^{2,\gamma}(\overline\Omega)$. By
Lemma~\ref{lem:zero_extension},
\begin{equation}\label{eq:q_prior}
\|E_0q_k\|_{H^r(\R^3)}\le C,
\qquad k=1,2.
\end{equation}

For $1<\mu\le\tau$, split the integral of
$|\widehat{E_0q_k}(\xi)|^2$ at $|\xi|=\mu$.
By \eqref{eq:q_entries} and \eqref{eq:fourier_bound_Q}, the
low-frequency integral is bounded by $C\mu^3E(\tau,\delta)^2$.
By \eqref{eq:q_prior}, the high-frequency integral is bounded
by $C\mu^{-2r}$. Plancherel's theorem gives
\begin{equation}\label{eq:frequency_split_bound}
\|q_k\|_{L^2(\Omega)}^2
=\|E_0q_k\|_{L^2(\R^3)}^2
\le C\bigl(\mu^3E(\tau,\delta)^2+\mu^{-2r}\bigr).
\end{equation}
For $0<\delta<\delta_0$, choose
\begin{equation}\label{eq:tau_choice}
\tau=\frac{|\log\delta|}{4R},
\qquad \mu=\tau^{2/(3+2r)}.
\end{equation}
For $\delta_0$ sufficiently small,
$\tau\ge\tau_0$ and $1<\mu\le\tau$. Since
$\tau^2\e^{2R\tau}\delta=\tau^2\delta^{1/2}\le C/\tau$,
we have $E(\tau,\delta)\le C/\tau$. Thus
\begin{equation}\label{eq:final_q_bound_with_exponent}
\|q_k\|_{L^2(\Omega)}^2
\le C\bigl(\mu^3\tau^{-2}+\mu^{-2r}\bigr)
\le C\tau^{-4r/(3+2r)}.
\end{equation}
Taking square roots and using $a\ge a_0$, we obtain
\eqref{eq:background_stability}.
\end{proof}

\subsection{Stability for the source}

We estimate the source difference in $H^{-2}(\Omega)$ and
interpolate with its a priori $H^s(\Omega)$ bound. Here
$H^{-2}(\Omega)$ denotes the dual of $H_0^2(\Omega)$, where
$H_0^2(\Omega)$ is the closure of $C_c^\infty(\Omega)$ in
$H^2(\Omega)$.

\begin{lemma}\label{lem:domain_interpolation}
Let $s>0$. There exists a constant $C=C(\Omega,s)>0$ such that
\begin{equation}\label{eq:source_interpolation}
\|G\|_{L^2(\Omega)}
\le C\|G\|_{H^{-2}(\Omega)}^{s/(s+2)}
       \|G\|_{H^s(\Omega)}^{2/(s+2)},
\qquad G\in H^s(\Omega).
\end{equation}
\end{lemma}

\begin{proof}
Extension by zero identifies $H_0^2(\Omega)$ with the closure
of $C_c^\infty(\Omega)$ in $H^2(\R^3)$, with equivalent norms.
By Sobolev duality, $(H_0^2(\Omega))'$ is identified with the
space of restrictions to $\Omega$ of distributions in $H^{-2}(\R^3)$.
The dual norm is equivalent to the quotient norm on this restriction
space \cite[Section~4.3]{ChandlerWildeHewettMoiola2015}.
Since $\Omega$ is smooth, Rychkov's extension theorem
\cite[Theorem~4.6]{ChandlerWildeHewettMoiola2015} provides a
single linear extension operator $\mathcal E_\Omega$ that is bounded
both from $H^{-2}(\Omega)$ to $H^{-2}(\R^3)$ and from
$H^s(\Omega)$ to $H^s(\R^3)$, with
$(\mathcal E_\Omega G)|_\Omega=G$.
Thus,
\[
\|\mathcal E_\Omega G\|_{H^{-2}(\R^3)}
\le C\|G\|_{H^{-2}(\Omega)},
\qquad
\|\mathcal E_\Omega G\|_{H^s(\R^3)}
\le C\|G\|_{H^s(\Omega)}.
\]
By the Fourier definition of the Sobolev norms and
H\"older's inequality,
\[
\|\mathcal E_\Omega G\|_{L^2(\R^3)}
\le
\|\mathcal E_\Omega G\|_{H^{-2}(\R^3)}^{s/(s+2)}
\|\mathcal E_\Omega G\|_{H^s(\R^3)}^{2/(s+2)}.
\]
Restricting to $\Omega$ proves \eqref{eq:source_interpolation}.
\end{proof}

Let $\wt u_\omega=u^0_{\omega,2}-u^0_{\omega,1}$ and
$\wt u_{2\omega}=u^0_{2\omega,2}-u^0_{2\omega,1}$.

\begin{proposition}\label{prop:source_stability}
Let $F_1,F_2\in\mathcal A_{M,s}$, and set
\[
\wt F=F_2-F_1,
\qquad \varepsilon=\|\wt u_\omega\|_{L^2(\Omega)}
+\|\wt u_{2\omega}\|_{L^2(\Omega)}.
\]
Then
\begin{equation}\label{eq:conditional_source_estimate}
\|\wt F\|_{L^2(\Omega)}\le C\varepsilon^{s/(s+2)}.
\end{equation}
The constant is independent of the admissible sources.
\end{proposition}

\begin{proof}
Subtracting the fundamental-frequency equations gives
\begin{equation}\label{eq:fundamental_error_pde}
(-\Delta-\omega^2)\wt u_\omega
+2a\bigl(u^0_{2\omega,2}\wt u_\omega
+u^0_{\omega,1}\wt u_{2\omega}\bigr)=\wt F.
\end{equation}
For $\phi\in C_c^\infty(\Omega)$, integration by parts gives
\begin{align*}
\langle\wt F,\phi\rangle
&=\int_\Omega \wt u_\omega(-\Delta-\omega^2)\phi\,dx\\
&\quad+2\int_\Omega a\bigl(u^0_{2\omega,2}\wt u_\omega
+u^0_{\omega,1}\wt u_{2\omega}\bigr)\phi\,dx.
\end{align*}
The uniform $L^\infty$ bounds for the background fields give
\[
|\langle\wt F,\phi\rangle|
\le C\bigl(\|\wt u_\omega\|_{L^2}
+\|\wt u_{2\omega}\|_{L^2}\bigr)\|\phi\|_{H^2}.
\]
By density, the estimate holds for every $\phi\in H_0^2(\Omega)$.
It follows that
\begin{equation}\label{eq:F_negative_domain}
\|\wt F\|_{H^{-2}(\Omega)}\le C\varepsilon.
\end{equation}
By the admissibility assumptions,
\[
\|\wt F\|_{H^s(\Omega)}\le2M.
\]
Lemma~\ref{lem:domain_interpolation}, applied to $\wt F$, gives
\[
\|\wt F\|_{L^2(\Omega)}
\le C\|\wt F\|_{H^{-2}(\Omega)}^{s/(s+2)}
       \|\wt F\|_{H^s(\Omega)}^{2/(s+2)}
\le C\varepsilon^{s/(s+2)}.
\]
\end{proof}

\begin{proof}[Proof of Theorem~\ref{thm:main_theorem}]
Lemma~\ref{lem:linearized_stability} gives
$\varepsilon\le C|\log\delta|^{-2r/(3+2r)}$.
Proposition~\ref{prop:source_stability} then gives
\eqref{eq:main_stability} with the stated exponent.
\end{proof}

\appendix
\section{Proof of the Quantitative Implicit Function Theorem}
\label{app:qift_proof}

\begin{proof}[Proof of Lemma~\ref{lem:qift}]
For $x\in\overline B_r^X(x_0)$, define
\begin{equation}\label{eq:app_gx}
\mathcal G_x(y)=y-A^{-1}\mathcal F(x,y).
\end{equation}
The fundamental theorem of calculus along the segment from $y_0$
to $y$ gives
\[
\mathcal G_x(y)-y_0
=-A^{-1}\mathcal F(x,y_0)
+\int_0^1\bigl[I-A^{-1}D_2\mathcal F(x,y_0+t(y-y_0))\bigr]
(y-y_0)\,dt.
\]
By \eqref{eq:qift_cond1}--\eqref{eq:qift_cond2},
\[
\|\mathcal G_x(y)-y_0\|_Y
\le\rho/2+\tfrac12\|y-y_0\|_Y\le\rho.
\]
Thus $\mathcal G_x$ maps the closed solution ball into itself.
Applying the same formula to $y_1$ and $y_2$ gives
\[
\|\mathcal G_x(y_1)-\mathcal G_x(y_2)\|_Y
\le\tfrac12\|y_1-y_2\|_Y.
\]
The contraction constant is independent of $x$. The Banach fixed
point theorem gives a unique $u(x)$ in the closed solution ball,
and $u(x_0)=y_0$.

For $x,x'$ in the parameter ball, the contraction estimate gives
\[
\|u(x')-u(x)\|_Y
\le2\|\mathcal G_{x'}(u(x))-\mathcal G_x(u(x))\|_Y.
\]
Continuity of $\mathcal F$ therefore implies continuity of $u$.
By \eqref{eq:qift_cond2}, a Neumann series shows that
$D_2\mathcal F(x,u(x))$ is invertible. The local implicit
function theorem at each interior point, together with continuity
and local uniqueness, shows that $u$ is $C^m$ on the open
parameter ball.
\end{proof}

\end{document}